\documentclass[11pt,letterpaper,twoside,reqno,nosumlimits]{amsart}

\usepackage[usenames,dvipsnames]{xcolor}
\usepackage{fancyhdr}
\usepackage{amsmath,amsfonts,amsbsy,amsgen,amscd,mathrsfs,amssymb,amsthm}
\usepackage{stackengine}
\usepackage{url}

\usepackage{caption}
\usepackage{subcaption}

\usepackage{mathtools}

\mathtoolsset{showonlyrefs}

\usepackage[font=small,margin=0.25in,labelfont={sc},labelsep={colon}]{caption}

\usepackage{tikz}
\usepackage{microtype}
\usepackage{enumitem}

\definecolor{dark-gray}{gray}{0.3}
\definecolor{dkgray}{rgb}{.4,.4,.4}
\definecolor{dkblue}{rgb}{0,0,.5}
\definecolor{medblue}{rgb}{0,0,.75}
\definecolor{rust}{rgb}{0.5,0.1,0.1}

\usepackage[colorlinks=true]{hyperref}

\hypersetup{urlcolor=rust}
\hypersetup{citecolor=dkblue}
\hypersetup{linkcolor=dkblue}

\usepackage{setspace}

\usepackage{graphicx}
\usepackage{booktabs,longtable,tabu} 
\usepackage{multirow} 
\usepackage{float}

\usepackage[full]{textcomp}

\usepackage[scaled=.98,sups,osf]{XCharter}\usepackage[scaled=1.04,varqu,varl]{inconsolata}\usepackage[type1]{cabin}\usepackage[utopia,vvarbb,scaled=1.07]{newtxmath}
\usepackage[cal=boondoxo]{mathalfa}
\usepackage[T1]{fontenc}

\usepackage[draft]{changes}
\definechangesauthor[color=blue]{ml}
\definechangesauthor[color=red]{jt}

\usepackage{bm} 
\graphicspath{{figures/}}

\newtheorem{theorem}{Theorem}[section]
\newtheorem{lemma}[theorem]{Lemma}

\newtheorem{proposition}[theorem]{Proposition}
\newtheorem{fact}[theorem]{Fact}

\theoremstyle{definition}

\newtheorem{example}[theorem]{Example}
\newtheorem{remark}[theorem]{Remark}
\newtheorem{assumption}[theorem]{Assumption}

\numberwithin{equation}{section} 
\numberwithin{figure}{section}
\numberwithin{table}{section}

\floatstyle{plaintop}
\newfloat{recipe}{thp}{lor}
\floatname{recipe}{Recipe}
\numberwithin{recipe}{section}

\providecommand{\mathbold}[1]{\bm{#1}}

\renewcommand{\phi}{\varphi}

\newcommand{\cnst}[1]{\mathrm{#1}} 
\newcommand{\econst}{\mathrm{e}}

\newcommand{\Id}{\mathbf{I}}

\providecommand{\mathbbm}{\mathbb} 
\newcommand{\R}{\mathbbm{R}}
\newcommand{\N}{\mathbbm{N}}

\newcommand{\Sym}{\mathbbm{H}}

\newcommand{\abs}[1]{\left\vert {#1} \right\vert}

\newcommand{\sgn}[1]{\operatorname{sgn}{#1}}

\newcommand{\diff}[1]{\mathrm{d}{#1}}
\newcommand{\idiff}[1]{\, \diff{#1}}

\newcommand{\grad}{\nabla} 

\newcommand{\Hess}{\operatorname{Hess}}

\newcommand{\Prob}[1]{\mathbb{P}\left\{{#1}\right\}}

\newcommand{\Expect}{\operatorname{\mathbb{E}}}
\newcommand{\E}{\Expect}

\DeclareMathOperator{\Var}{\bf{Var}}

\newcommand{\dirform}{\bm{\mathcal{E}}}

\stackMath \setstackgap{S}{1pt}

\newcommand{\vct}[1]{\mathbold{#1}}
\newcommand{\mtx}[1]{\mathbold{#1}}

\newcommand{\rank}{\operatorname{rank}}

\newcommand{\intdim}{\operatorname{intdim}}

\newcommand{\trace}{\operatorname{tr}}
\newcommand{\ntr}{\operatorname{\bar{\trace}}}

\newcommand{\psdle}{\preccurlyeq}
\newcommand{\psdge}{\succcurlyeq}

\newcommand{\ip}[2]{\left\langle {#1},\ {#2} \right\rangle}

\newcommand{\norm}[1]{\left\Vert {#1} \right\Vert}

\allowdisplaybreaks

\title[From Poincar{\'e} to Matrix Concentration]{From Poincar{\'e} Inequalities \\ to Nonlinear Matrix Concentration}

\author[D.~Huang and J.~A.~Tropp]{De Huang and Joel~A.~Tropp}

\date{26 April 2020.  Revised 11 June 2020 and 10 October 2020.}


\begin{document}

\begin{abstract}
This paper deduces exponential matrix concentration
from a Poincar{\'e} inequality via a short, conceptual argument.
Among other examples, this theory applies to matrix-valued functions of a
uniformly log-concave random vector. The proof relies on the subadditivity of Poincar{\'e} inequalities and a
chain rule inequality for the trace of the matrix Dirichlet form.
It also uses a symmetrization technique to avoid difficulties associated
with a direct extension of the classic scalar argument.
\end{abstract}

\maketitle

\section{Introduction}

Matrix concentration inequalities describe the probability
that a random matrix is close to its expected value, with
deviations measured by the $\ell_2$ operator norm.  These results have
had a profound impact on a wide range of areas in computational mathematics
and statistics.  See the monograph~\cite{Tro15:Introduction-Matrix}
for an introduction to the subject and its applications.

As did the field of scalar concentration, matrix concentration theory
began with simple models, such as sums of independent
random matrices~\cite{LP86:Inegalites-Khintchine,Rud99:Random-Vectors,AW02:Strong-Converse,Tro12:User-Friendly}
and matrix-valued martingale sequences~\cite{PX97:Noncommutative-Martingale,Oli09:Concentration-Adjacency,Tro11:Freedmans-Inequality}.
In recent years, researchers have sought to develop results that hold for
a wider class of random matrix models.  Some initial successful efforts were based on
exchangeable pairs techniques~\cite{MJCFT14:Matrix-Concentration,PMT16:Efron-Stein-Inequalities},
but these methods do not address all cases of interest.

Researchers have also tried to extend scalar
concentration techniques based on functional inequalities. An early attempt, by Chen \& Tropp~\cite{CT14:Subadditivity-Matrix}, demonstrates
that (traces of) matrix variance and entropy quantities are subadditive, which 
leads to some Poincar{\'e} and modified log-Sobolev inequalities.  A number of authors,
including~\cite{CH16:Characterizations-Matrix,CHT17:Exponential-Decay,CH19:Matrix-Poincare},
have pursued this line of work.
Unfortunately, these approaches have not
been sufficient to reproduce all the results that have been established in the simpler models.

Very recently, Aoun et al.~\cite{ABY19:Matrix-Poincare} have
shown that a matrix form of the Poincar{\'e} inequality implies
subexponential concentration of a random matrix with respect to the $\ell_2$
operator norm.  We believe that this is the first instance
where a matrix functional inequality leads unconditionally
to a matrix concentration result (with respect to the operator norm).
Nevertheless, it remains an open question to deduce a full spectrum
of matrix concentration results from matrix functional inequalities.

In this paper, we improve on~\cite{ABY19:Matrix-Poincare}
by demonstrating that the ordinary scalar Poincar{\'e}
inequality also leads to subexponential concentration with respect
to the operator norm.
Our argument has some elements in common with the work in
\cite{ABY19:Matrix-Poincare}, but we have found a route to avoid
most of the technical difficulty of their approach.

The basic idea is to bound the trace of an odd function
(for example, the hyperbolic sine)
of the random matrix using a symmetrization argument.
The variance appears, and 
the Poincar{\'e} inequality
yields a bound on the variance in terms of the
Dirichlet form.  Last,
we apply a new matrix chain rule inequality
for the Dirichlet form to
obtain a moment comparison.
After this paper was written, we learned
that Bobkov \& Ledoux proposed
a similar argument in the scalar
case~\cite[Sec.~4]{BL97:Poincares-Inequalities}.

As in the scalar setting, Poincar{\'e} inequalities may not capture
the strongest concentration properties that are possible.  In a companion
paper~\cite{HT20:Matrix-Concentration}, we demonstrate that
\emph{local} Poincar{\'e} inequalities lead to the optimal
subgaussian concentration results.  The analysis in~\cite{HT20:Matrix-Concentration}
captures most of the previous results on matrix concentration,
but it involves more technical machinery.  We also refer the
reader to work of Junge \& Zeng~\cite{JZ15:Noncommutative-Martingale}
that contains similar results in the fully noncommutative setting.

An expert reader may still wonder about the role of (modified)
log-Sobolev inequalities in establishing matrix concentration
inequalities.  At the time of writing, it is not clear how
to obtain a matrix analog of the log-Sobolev inequality
that would imply matrix concentration results in the
same spirit as the ones in this paper or the related
works~\cite{JZ15:Noncommutative-Martingale,ABY19:Matrix-Poincare,HT20:Matrix-Concentration}.

\section{Main result}

This section summarizes our notation and the setting for our problem.
It highlights our primary result on matrix concentration, and it
gives a number of examples.  In the next section,
we comment on the relationship with previous work.

\subsection{Notation}

Let $\Sym_d$ be the real linear space of $d \times d$ self-adjoint complex matrices,
equipped with the $\ell_2$ operator norm $\norm{\cdot}$.
We work with both the ordinary trace, $\trace$, and the normalized trace, $\ntr := d^{-1} \trace$
on the space $\Sym_d$.
Matrices, and occasionally vectors, are written in boldface italic.
In particular, $\mtx{f}$ and $\mtx{g}$ refer to functions taking values in $\Sym_d$.
The cone $\Sym_d^+$ contains the positive semidefinite matrices,
and the symbol $\psdle$ refers to the semidefinite order.

Given a function $\phi : \R \to \R$ taking real values,
we extend it to a function $\phi : \Sym_d \to \Sym_d$
on self-adjoint matrices by means of the spectral resolution:
$$
\mtx{A} = \sum_{\lambda \in \mathrm{spec}(\mtx{A})} \lambda \, \mtx{P}_{\lambda} \in \Sym_d
\quad\text{implies}\quad
\phi(\mtx{A}) = \sum_{\lambda \in \mathrm{spec}(\mtx{A})} \phi(\lambda) \, \mtx{P}_{\lambda} \in \Sym_d.
$$
Whenever we apply a scalar function, such as a power
or a hyperbolic function, to a matrix, we are referring
to the standard matrix function.  Nonlinear functions
bind before the trace.

We use familiar notation from probability. The operator $\E$
returns the expectation, and $\Prob{\cdot}$ is the probability
of an event.  The symbol $\sim$ means ``has the distribution.''
Nonlinear functions bind before the expectation.

\subsection{Random matrices}

Let $\Omega$ be a Polish space, equipped with a probability measure $\mu$,
and write $\E_{\mu}$ for the integral with respect to the measure $\mu$.
Consider a $\mu$-integrable matrix-valued function $\mtx{f} : \Omega \to \Sym_d$ on the state space $\Omega$.
By drawing a random variable $Z \sim \mu$, we can construct a random matrix
$\mtx{f}(Z)$.  Our goal is to understand the concentration of $\mtx{f}(Z)$
around its mean $\E_{\mu} \mtx{f}$.

\begin{example}[Gaussians]
Consider the Gaussian space $(\R^n, \gamma_n)$ of $n$-dimensional
real vectors equipped with the standard normal measure $\gamma_n$.
Suppose we are interested in a matrix-valued function  $\mtx{f}(\vct{X})$
of a standard normal random vector $\vct{X} \sim \gamma_n$.
A familiar example~\cite[Chap.~5]{Tro15:Introduction-Matrix} is the matrix Gaussian series
\begin{equation} \label{eqn:gauss-series}
\mtx{f}(\vct{X}) = \sum_{i=1}^n X_i \mtx{A}_i
\quad\text{where $\mtx{X} \sim \gamma_n$ and $\mtx{A}_1, \dots, \mtx{A}_n \in \Sym_d$ are fixed.}
\end{equation}
We will use the Gaussian case as a running example to illustrate
the concepts that we introduce.
\end{example}

\subsection{Markov processes}

Suppose that we can identify an ergodic, reversible, time-homogeneous Markov process
$(Z_t : t \geq 0) \subset \Omega$ with initial
value $Z_0$ and stationary measure $\mu$.
This induces a matrix-valued Markov process
$( \mtx{f}(Z_t) : t \geq 0 ) \subset \Sym_d$.

By ergodicity, for any point $z \in \Omega$,
we have the limit $\E[ \mtx{f}(Z_t) \, | \, Z_0 = z ] \to \E_{\mu} \mtx{f}$
as $t \to \infty$.
The results in this paper build on
the intuition that a random matrix $\mtx{f}(Z)$
with $Z \sim \mu$ concentrates sharply about
its mean when the matrix-valued Markov process
tends quickly to equilibrium.

\begin{example}[Gaussians]
We can construct a reversible Markov process $(\vct{X}_t : t \geq 0) \subset \R^n$,
called the \emph{Ornstein--Uhlenbeck (OU) process},
by means of the stochastic differential equation
$$
\diff{\vct{X}}_t = - {\vct{X}}_t \idiff{t} + \sqrt{2} \idiff{\vct{B}}_t
\quad\text{with initial value $\vct{X}_0 \in \R^n$,}
$$
where $( \vct{B}_t : t \geq 0 ) \subset \R^n$ is random vector whose coordinates are
independent Brownian motions.  The stationary measure of
the OU process is the standard normal distribution $\gamma_n$.
\end{example}

\subsection{Derivatives and energy}

To understand how quickly a matrix Markov process $\mtx{f}(Z_t)$
converges to stationarity, we introduce notions
of the ``squared derivative'' and the ``energy'' of the function $\mtx{f}$.

Inspired by~\cite{ABY19:Matrix-Poincare}, 
we define the matrix \emph{carr{\'e} du champ operator} by the formula
\begin{equation} \label{eqn:carre-limit}
\mtx{\Gamma}(\mtx{f})(z) := \lim_{t \downarrow 0} \frac{1}{2t} \E \big[ \big(\mtx{f}(Z_t) - \mtx{f}(Z_0) \big)^2 \,\big|\, Z_0 = z \big] \in \Sym_d^+
\quad\text{for $z \in \Omega$.}
\end{equation}
In many instances, the carr{\'e} du champ $\mtx{\Gamma}(\mtx{f})$ has a natural interpretation
as a squared derivative of $\mtx{f}$.
The expectation of the carr{\'e} du champ is called the \emph{matrix Dirichlet form}:
\begin{equation} \label{eqn:dirichlet-limit}
\dirform(\mtx{f}) := \lim_{t \downarrow 0} \frac{1}{2t} \E_{Z \sim \mu} \big[ \big(\mtx{f}(Z_t) - \mtx{f}(Z_0)\big)^2 \,\big|\, Z_0 = Z \big] \in \Sym_d^+.
\end{equation}
The Dirichlet form $\dirform(\mtx{f})$ reflects the total energy of the function $\mtx{f}$.

In a general setting, it requires some care to make sense of the
definitions~\eqref{eqn:carre-limit} and~\eqref{eqn:dirichlet-limit}.
Without further comment, we restrict our attention to a ``nice'' class of functions
where the limit in~\eqref{eqn:carre-limit} exists pointwise and in $L_1(\mu)$
and where calculus operations are justified.
By approximation, our main results on concentration hold for a wider
class of functions.

\begin{example}[Gaussians]
According to~\cite[Prop.~5.5]{ABY19:Matrix-Poincare},
the matrix carr{\'e} du champ operator and matrix Dirichlet form of the OU process
are determined by
$$
\mtx{\Gamma}(\mtx{f})(\vct{x}) = \sum_{i=1}^n (\partial_i \mtx{f}(\vct{x}))^2
	\quad\text{for $\vct{x} \in \R^n$}
\quad\text{and}\quad
\dirform(\mtx{f}) = \sum_{i=1}^n \E_{\gamma_n} (\partial_i \mtx{f})^2.
$$
The interpretations as a squared derivative and an energy are evident,
and it is easy to check when the carr{\'e} du champ is defined.

The matrix Gaussian series~\eqref{eqn:gauss-series} provides
an illustration:
$$
\mtx{\Gamma}(\mtx{f}) = \dirform(\mtx{f}) = \sum_{i=1}^n \mtx{A}_i^2.
$$
This quantity is familiar from work on matrix concentration for Gaussian series~\cite[Chap.~4]{Tro15:Introduction-Matrix}.
\end{example}

\subsection{Trace Poincar{\'e} inequalities}

The \emph{matrix variance} of a function $\vct{f} : \Omega \to \Sym_d$
with respect to the distribution $\mu$ is defined as
\begin{equation}\label{eqn:matrix-var}
\Var_{\mu}[ \mtx{f} ] :=
	\E_{\mu}\big[ (\mtx{f} - \E_{\mu}\mtx{f} )^2 \big]
	= \E_{\mu}[ \vct{f}^2 ] - (\E_{\mu}  \vct{f} )^2
	\in \Sym_d^+.
\end{equation}
As in the scalar case, the variance reflects fluctuations
of the random matrix $\mtx{f}(Z)$ about its mean, where
the random variable $Z \sim \mu$.

We say that the Markov process satisfies a \emph{trace Poincar{\'e} inequality}
with constant $\alpha > 0$ if
\begin{equation} \label{eqn:trace-poincare}
\trace \Var_{\mu}[ \mtx{f} ]
	\leq \alpha \cdot \trace \dirform(\mtx{f})
	\quad\text{for all $\mtx{f} : \Omega \to \Sym_d$.} \end{equation}
In other words, the trace variance of $\mtx{f}(Z)$ is controlled
by the energy in the function $\mtx{f}$.  The inequality~\eqref{eqn:trace-poincare}
provides a way to quantify the ergodicity of the Markov process.

As it happens, the trace Poincar{\'e} inequality is equivalent to
an ordinary Poincar{\'e} inequality.  We are grateful to Ramon
Van Handel for this observation.  The same result has recently
appeared in the independent work of Garg et al.~\cite{GKS20:Scalar-Poincare}.

\begin{proposition}[Equivalence of Poincar{\'e} inequalities] \label{prop:equiv}
Consider a Markov process $(Z_t : t \geq 0) \subset \Omega$
with stationary measure $\mu$.  The following are equivalent:

\begin{enumerate}
\item	\textbf{Scalar Poincar{\'e}.}  For all $f : \Omega \to \R$, it holds that
$\Var_{\mu}[ f ] \leq \alpha \cdot \dirform( f )$.

\item	\textbf{Trace Poincar{\'e}.}  For all $d \in \N$ and all $\mtx{f} : \Omega \to \Sym_d$,
it holds that $\trace \Var_{\mu}[ \mtx{f} ] \leq \alpha \cdot \trace \dirform( \mtx{f})$.
\end{enumerate}

\noindent
The Poincar{\'e} constant $\alpha \geq 0$ is the same for both cases.
\end{proposition}

\begin{proof}
It is evident that the validity of the trace Poincar{\'e} inequality for all $d \in \N$
implies the scalar Poincar{\'e} inequality.  For the reverse implication, it suffices to
consider a real matrix-valued function $\mtx{f} : \Omega \to \Sym_d(\R)$ with zero mean.
For vectors $\vct{u}, \vct{v} \in \R^d$, define the scalar function
$g(z) = \ip{ \vct{u} }{ \mtx{f}(z) \, \vct{v} } \in \R$.  Apply the scalar
Poincar{\'e} inequality to $g$ and invoke the definition~\eqref{eqn:dirichlet-limit}
of the Dirichlet form.  Thus,
\begin{align*}
\E_{\mu} \ip{ \vct{u} }{ \mtx{f}(z) \, \vct{v} }^2
	\leq \alpha \cdot \lim_{t \downarrow 0} \frac{1}{2t} \E_{Z \sim \mu} \big[
	\ip{ \vct{u} }{ (\mtx{f}(Z_t) - \mtx{f}(Z_0)) \, \vct{v} }^2 \, \big\vert\, Z_0 = Z \big].
\end{align*}
Instate this inequality with $\vct{v} = \mathbf{e}_i$ for each $i = 1, \dots, d$
and sum over $i$ to arrive at
$$
\ip{\vct{u}}{ \smash{\Var_{\mu}[ \mtx{f} ]} \, \vct{u}}
	\leq \ip{ \vct{u} }{ \dirform(\mtx{f}) \, \vct{u} }.
$$
Average over $\vct{u} \sim \textsc{uniform}\{\pm 1\}^d$ to reach the trace Poincar{\'e}
inequality.  To extend this argument to complex matrices, apply the same approach
to the real and imaginary parts of the inner product.
\end{proof}

The main result of this paper is that concentration properties
of the random matrix $\mtx{f}(Z)$ follow from the
trace Poincar{\'e} inequality~\eqref{eqn:trace-poincare}
or, equivalently, the scalar Poincar{\'e} inequality.

\begin{example}[Gaussians]
It is well known that the OU process satisfies the Poincar{\'e}
inequality with constant $\alpha = 1$.  Thus, it satisfies
the trace Poincar{\'e} inequality~\eqref{eqn:trace-poincare}
with $\alpha = 1$.  For an alternative proof, see~\cite[Thm.~1.2]{ABY19:Matrix-Poincare}.
\end{example}

\subsection{Subexponential concentration and expectation bounds}

We are now prepared to present our main result.
It demands several hypotheses,
which will be enforced throughout the paper.

\begin{assumption}[Conditions] \label{ass:main}
We assume that
\begin{enumerate}
\item	The Markov process $(Z_t : t \geq 0) \subset \Omega$ is reversible
and homogeneous, with initial value $Z_0$ and stationary measure $\mu$.
\item	The process admits a trace Poincar{\'e} inequality~\eqref{eqn:trace-poincare}
with constant $\alpha$.  Equivalently, the process admits a scalar Poincar{\'e}
inequality with the same constant $\alpha$.

\item	The class of valid functions is suitably restricted so that manipulations of
expectations, limits, and derivatives are justified.
\end{enumerate}
\end{assumption}

Under Assumption~\ref{ass:main},
we will deduce 
subexponential concentration of the random matrix $\mtx{f}(Z)$
around its mean $\E_{\mu} \mtx{f}$, where we measure the size of deviations
with the $\ell_2$ operator norm $\norm{\cdot}$.  

\begin{theorem}[Subexponential Concentration] \label{thm:main}
Enforce Assumption~\ref{ass:main}.
Let $\mtx{f} : \Omega \to \Sym_d$ be a matrix-valued function, and define the variance proxy
\begin{equation*} \label{eqn:vf}
v_{\mtx{f}} := \norm{ \norm{ \mtx{\Gamma}(\mtx{f})(z) } }_{L_{\infty}(\mu)}.
\end{equation*}
For all $\lambda > 0$,
\begin{equation} \label{eqn:main-tail}
\mathbbm{P}_{\mu}\big\{ \norm{ \smash{\mtx{f} - \Expect_{\mu} \mtx{f}} } \geq \sqrt{\alpha v_{\mtx{f}}} \cdot \lambda  \big\}
	\leq 6 d \cdot \econst^{- \lambda}.
	\end{equation}
In particular,
\begin{equation} \label{eqn:main-expect}
\E_{\mu} \norm{ \smash{ \mtx{f} - \Expect_{\mu} \mtx{f} } }
	\leq \log( 6 \econst d ) \cdot \sqrt{\alpha v_{\mtx{f}}}.
	\end{equation}
\end{theorem}

\noindent
The proof of Theorem~\ref{thm:main} appears in Section~\ref{sec:exp-moments}
after we present some more background on matrix-valued Markov processes.
Note that we have made no effort to refine constants.

The main point of the tail bound~\eqref{eqn:main-tail} is that
the random matrix $\mtx{f}(Z)$ exhibits exponential concentration on the scale $\sqrt{\alpha v_{\mtx{f}}}$.
The variance proxy $v_{\mtx{f}}$ is analogous to a global bound on
the Lipschitz constant of $\mtx{f}$.
Be aware that we cannot achieve tail decay faster than exponential
under the sole assumption of a Poincar{\'e} inequality, so this
approach may not capture the strongest possible concentration.
The leading constant in~\eqref{eqn:main-tail} reflects
the ambient dimension $d$ of the matrix;
this feature is typical of matrix concentration bounds.

The expectation bound~\eqref{eqn:main-expect} shows
that the average value of $\norm{ \smash{\mtx{f} - \E_{\mu} \mtx{f}} }$ is proportional to the
square root $\sqrt{v_{\vct{f}}}$ of the variance proxy and to the logarithm of the ambient dimension.
For many examples, the optimal bound contains the \emph{square-root} of the logarithm,
but the result~\eqref{eqn:main-expect} is nontrivial and informative.

\begin{example}[Gaussian Series]
Theorem~\ref{thm:main} applies to a matrix-valued function of a standard normal vector.
For instance, according to~\eqref{eqn:main-expect}, a matrix Gaussian series satisfies the expectation bound
$$
\E \norm{ \sum_{i=1}^n X_i \mtx{A}_i } \leq \log(6 \econst d) \cdot \norm{ \sum_{i=1}^n \mtx{A}_i^2 }^{1/2}.
$$
Modulo the constant and the power on the logarithm, this bound is qualitatively correct for worst-case examples. On the other hand, the subexponential tail bound~\eqref{eqn:main-tail} does not reproduce the actual
subgaussian behavior.  See~\cite[Chap.~5]{Tro15:Introduction-Matrix} for discussion.
Example~\ref{ex:chaos} describes an application to Gaussian chaos that requires
tools that are more delicate than Theorem~\ref{thm:main}.
\end{example}

\begin{remark}[Extensions]
The bounds in Theorem~\ref{thm:main} can be refined in several ways.
We can replace the variance proxy $v_{\mtx{f}}$ with less stringent measures of the size of the carr{\'e}
du champ $\mtx{\Gamma}(\mtx{f})$, such as the expected Schatten $q$-norm with $q \approx \log d$. 
It is also possible to replace the ambient dimension $d$ with a measure
of the intrinsic dimension of the random matrix $\mtx{f}(Z)$. 
See Section~\ref{sec:poly-moments}.
\end{remark}

\begin{remark}[Rectangular case]
By a standard formal argument, we can extend all the results here to
a function $\mtx{h}: \Omega \to \mathbb{M}_{d_1 \times d_2}$ that takes
values in the $d_1 \times d_2$ complex matrices.
To do so, we simply apply our results to the self-adjoint function
$$
\mtx{f}(z) = \begin{bmatrix} \mtx{0} & \mtx{h}(z) \\ \mtx{h}(z)^* & \mtx{0} \end{bmatrix}
	\in \Sym_{d_1 + d_2}
\quad\text{for $z \in \Omega$.}
$$
See~\cite[Sec.~2.1.17]{Tro15:Introduction-Matrix} for details.
\end{remark}

\subsection{Examples}

To indicate the scope of Theorem~\ref{thm:main},
let us present some more examples.  Most of these
examples are actually known to exhibit subgaussian
matrix concentration, but there is
at least one case (Section~\ref{sec:scp})
where the results here are currently the best available.

\subsubsection{Log-concave measures}

The Gaussian case is a particular example of a more general result for log-concave measures.
Suppose that $J : \R^n \to \R$ is a strongly convex function that satisfies
$\Hess J \psdge \eta \Id$ uniformly.  Construct the probability
measure $\mu$ on $\R^n$ whose density is proportional to $\econst^{-J}$.
In this example, we briefly discuss concentration of matrix-valued functions
$\mtx{f}(\vct{X})$ where $\vct{X} \sim \mu$.  This model is interesting
because it captures a type of negative dependence.

The appropriate Markov process $(\vct{X}_t : t \geq 0)$ evolves with
the stochastic differential equation
$$
\diff{\vct{X}}_t = - \grad J(\vct{X}_t) \idiff{t} + \sqrt{2} \idiff{\vct{B}}_t
\quad\text{with initial value $\vct{X}_0 \in \R^n$,}
$$
where $( \vct{B}_t : t \geq 0 ) \subset \R^n$ is Brownian motion.
The stationary distribution is $\mu$, and the matrix carr{\'e} du champ is $$
\mtx{\Gamma}(\mtx{f})(\vct{x}) = \sum_{i=1}^n (\partial_i \mtx{f}(\vct{x}))^2
	\quad\text{for $\vct{x} \in \R^n$.}
$$
It is well known that these diffusions satisfy a Poincar{\'e} inequality
with constant $\alpha = 1/\eta$; see~\cite[Cor.~4.8.2]{BGL14:Analysis-Geometry}.
Therefore, Theorem~\ref{thm:main} applies.

\textit{A fortiori}, these log-concave measures also satisfy a local Poincar{\'e} inequality,
which leads to subgaussian matrix concentration~\cite[Sec.~2.12.2]{HT20:Matrix-Concentration}.

\subsubsection{Riemannian manifolds with positive curvature}

More generally, let $(M, \mathfrak{g})$ be a compact Riemannian manifold with co-metric $\mathfrak{g}$.
The manifold carries a canonical Riemannian probability measure $\mu_\mathfrak{g}$.
The diffusion whose infinitesimal generator is the Laplace--Beltrami operator
$\Delta_{\mathfrak{g}}$ is called the Brownian motion on the manifold.  This
is a reversible, ergodic Markov process.  Its matrix carr{\'e} du champ takes
the form
\begin{equation} \label{eqn:manifold-gamma}
\mtx{\Gamma}(\mtx{f})(z) = \sum_{i,j} g^{ij}(z) \, (\partial_i \mtx{f}(z)) \, (\partial_j \mtx{f}(z))
\quad\text{for $\mtx{f} : \Omega \to \Sym_d$.}
\end{equation}
The co-metric $g$ and the partial derivatives $\partial_i \mtx{f}$ are computed
with respect to local coordinates.  See~\cite{BGL14:Analysis-Geometry} for
an introduction to diffusions on manifolds; the companion paper~\cite{HT20:Matrix-Concentration}
treats matrix-valued functions on manifolds.

Suppose that the manifold has uniformly positive Ricci curvature, where the Ricci tensor
has eigenvalues bounded below by $\rho$.  By now, it is a classic fact that the
Brownian motion on this manifold satisfies the scalar Poincar{\'e} inequality
with constant $\alpha = \rho^{-1}$.  See~\cite[Sec.~4.8]{BGL14:Analysis-Geometry}.

\begin{example}[Sphere]
For $n \geq 2$, the unit sphere $\mathbb{S}^{n} \subset \R^{n+1}$ is a Riemannian submanifold
of $\R^{n+1}$.  Its canonical measure is the uniform probability distribution,
and the carr{\'e} du champ of the Brownian motion on the sphere
is computed using~\eqref{eqn:manifold-gamma}.
The sphere has positive Ricci curvature with $\rho = n - 1$,
so it admits a Poincar{\'e} inequality with $\alpha = (n-1)^{-1}$.
Thus, matrix-valued functions on the sphere satisfy exponential matrix
concentration.
\end{example}

A similar story can be told about every positively curved manifold.
In fact, in this setting, we even have subgaussian matrix concentration
because of the stronger arguments in~\cite{HT20:Matrix-Concentration}.

\subsubsection{Products}

Consider a probability space $(\Omega, \mu)$.  It is common to work with multivariate
functions defined on the product space $(\Omega^n, \mu^{\otimes n})$.  There is
a standard construction~\cite[Sec.~2.3.2]{VH16:Probability-High} of a
Markov process on the product space.  Aoun et al.~\cite{ABY19:Matrix-Poincare}
verify that the (matrix) carr{\'e} du champ of this process is
$$
\mtx{\Gamma}(\mtx{f})(\vct{z})
	= \frac{1}{2} \sum_{i=1}^n \E_{Z \sim \mu} \big[ \big( \mtx{f}(z_1, \dots, z_n) - \mtx{f}(z_1, \dots, z_{i-1}, Z,  z_{i+1}, \dots, z_n) \big)^2 \big].
$$
This is the sum of squared discrete derivatives, each averaging over perturbations in a single coordinate.
The variance proxy $v_{\mtx{f}}$ takes the form
$$
v_{\mtx{f}} = \operatorname{ess\,sup}\nolimits_{\mtx{z} \in \Omega^n} \frac{1}{2} \norm{ \sum_{i=1}^n \E_{Z \sim \mu} \big[ \big( \mtx{f}(z_1, \dots, z_n) - \mtx{f}(z_1, \dots, z_{i-1}, Z,  z_{i+1}, \dots, z_n) \big)^2 \big] }
$$
The variance proxy coincides with the matrix bounded difference that arises in Paulin et al.~\cite{PMT16:Efron-Stein-Inequalities}.
Aoun et al.~prove that the Markov process satisfies a trace Poincar{\'e} inequality~\eqref{eqn:trace-poincare}
with constant $\alpha = 1$.  Therefore, Theorem~\ref{thm:main} yields a suboptimal version
of the exponential matrix Efron--Stein inequality~\cite[Thm.~4.3]{PMT16:Efron-Stein-Inequalities}.
See~\cite{ABY19:Matrix-Poincare,HT20:Matrix-Concentration} for more details.

\begin{remark}[Bernstein concentration?]
In the scalar case, Bobkov \& Ledoux~\cite[Cor.~3.2]{BL97:Poincares-Inequalities} have
shown that a function $f : \Omega^n \to \R$ on a product space $(\Omega^n, \mu^{\otimes n})$
exhibits Bernstein-type concentration when $\mu$ admits a Poincar{\'e} inequality  with constant $\alpha$.
In detail,
$$
\Prob{ \abs{f  - \E f} > t } \leq \exp\left( \frac{-\mathrm{c} t^2}{v + Bt} \right)
\quad\text{with}\quad
v = \norm{ \sum_{i=1}^n (\partial_i f)^2 }_{L_{\infty}} \quad\text{and}\quad
B = \max\nolimits_i \norm{ \partial_i f }_{L_{\infty}}.
$$
The constant $\cnst{c}$ depends on the constant $\alpha$ in the Poincar{\'e} inequality.
Our approach does not seem powerful enough to transfer this insight to the matrix setting.
Nevertheless, our paper~\cite{HT20:Matrix-Concentration} demonstrates
that a \emph{local} Poincar{\'e} inequality is sufficient to achieve
Bernstein concentration.
\end{remark}

\subsubsection{Stochastic covering property}
\label{sec:scp}

Aoun et al.~\cite{ABY19:Matrix-Poincare} have considered a model for negatively
dependent functions on the hypercube $\{0, 1\}^n$, namely the
class of measures with the \emph{stochastic covering property} (SCP).
For a $k$-homogeneous measure $\mu$ with the SCP, it is possible to construct a
Markov process that satisfies the trace Poincar{\'e} inequality~\eqref{eqn:trace-poincare}
with constant $2k$.  Thus, Theorem~\ref{thm:main} applies.
See~\cite{PP14:Concentration-Lipschitz,HS19:Modified-Log-Sobolev,ABY19:Matrix-Poincare}
for a more complete discussion of this example.

\begin{remark}[Subgaussian concentration?]
Although the Markov process associated with an SCP measure
satisfies a log-Sobolev inequality, we do not know
if it satisfies the local Poincar{\'e} inequality that we would need to activate
the subgaussian concentration inequalities in~\cite{HT20:Matrix-Concentration}.
\end{remark}

\subsubsection{Expander graphs}

Let $G = (\Omega, E)$ be a $k$-regular, connected, undirected graph.
We can construct a Markov process $(Z_t : t \geq 0) \subset \Omega$
by taking a continuous-time random walk on the vertex set $\Omega$.
The stationary measure $\mu$ of the random walk is the uniform
measure on $\Omega$.  The carr{\'e} du champ operator takes the form
$$
\mtx{\Gamma}(\mtx{f})(z) = \frac{1}{2k} \sum_{(z',z) \in E} \big(\mtx{f}(z') - \mtx{f}(z)\big)^2.
$$
In other words, the ``squared gradient'' is just the half the average squared
difference between the matrix at the current vertex and its $k$ neighbors.
$$
v_{\mtx{f}} := \max_{z \in \Omega} \norm{ \frac{1}{2k} \sum_{(z',z)\in E} \big(\mtx{f}(z') - \mtx{f}(z) \big)^2 }.
$$
The variance proxy is just the maximum ``squared gradient''
at any vertex.

Assume that the Markov process satisfies the (scalar) Poincar{\'e}
inequality, Proposition~\ref{prop:equiv}(1), with constant $\alpha$.
In this case, we say that $G$ is an $\alpha$-expander graph.
According to Theorem~\ref{thm:main}, a matrix-valued function
$\mtx{f} : \Omega \to \Sym_d$ on an $\alpha$-expander graph satisfies a
subexponential matrix concentration inequality:
$$
\Prob{ \norm{ \smash{\mtx{f} - \E_\mu \mtx{f}} }\geq \sqrt{\alpha v_{\mtx{f}}} \cdot \lambda }
	\leq 6d \cdot \econst^{-\lambda}.
$$
Therefore, local control over the fluctuations yields global concentration around the mean.
The number of vertices where the function departs from its mean is controlled by
the quality $\alpha$ of the expander.

This example is potentially interesting because the Markov process
does not satisfy a dimension-independent (modified)
log-Sobolev inequality~\cite[Sec.~5]{BT06:Modified-Logarithmic}.
Indeed, to achieve the functional inequality
$$
\mathrm{Ent}_{\mu}(f) \leq
	\beta \cdot \mathcal{E}(f, \log f)
	\quad\text{for all $f : \Omega \to \R$},
$$
it is necessary that $\beta \geq \mathrm{const} \cdot \log( \# \Omega )$.
See~\cite[Sec.~4.3]{Mun19:Li-Yau-Inequality} for related
results on curvature-dimension conditions of Bakry--{\'E}mery type.

\begin{remark}[Matrix Expander Chernoff]
Although the modified log-Sobolev inequality fails,
it is still possible to establish subgaussian concentration
for the ergodic averages of a matrix-valued random walk
on an expander graph~\cite{GLSS18:Matrix-Expander}.
\end{remark}

\section{Related work}

\subsection{Markov processes}

Much of the classical research on Markov processes concerns the relationship between
the geometry of the state space and the behavior of
canonical diffusion processes (e.g., Brownian motion
on a Riemannian manifold).
For an introduction, we recommend
the lecture notes~\cite{VH16:Probability-High}.  A more comprehensive
source is the treatise~\cite{BGL14:Analysis-Geometry}.  

Matrix-valued Markov processes first arose in the mathematical
physics literature as a model for the evolution of a quantum system.
Some of the foundational works include Davies~\cite{Dav69:Quantum-Stochastic}
and Lindblad~\cite{Lin76:Generators-Quantum}.
Quantum information theory has provided a new impetus
for studying matrix-valued Markov processes;
see~\cite{KT13:Quantum-Logarithmic}
for a discussion and some background references.

Here, we are interested in a mixed classical-quantum
setting, where a classical Markov process drives
a matrix-valued function.  Surprisingly, this model
does not seem to have received much attention
until the last few years.
See Cheng et al.~\cite{CHT17:Exponential-Decay}
for a more expansive framework that includes this case.
Other foundational results appear
in~\cite{ABY19:Matrix-Poincare,HT20:Matrix-Concentration}.

\subsection{Functional inequalities}

In the scalar setting, the connection between functional inequalities,
convergence of Markov processes, and concentration
is a long-standing topic of research.  References
include~\cite{Led01:Concentration-Measure,BLM13:Concentration-Inequalities,
BGL14:Analysis-Geometry,VH16:Probability-High}.

Functional inequalities for matrices were
originally formulated in the mathematical physics literature;
for example, see the work of Gross~\cite{Gro75:Hypercontractivity-Logarithmic}.
The application of functional inequalities to the
ergodicity of matrix-valued Markov processes dates
back at least as far as the
papers~\cite{MOZ98:Dissipative-Dynamics,OZ99:Hypercontractivity-Noncommutative}.

Functional inequalities in the mixed classical-quantum setting
seem to have a more recent vintage.  Chen \& Tropp~\cite{CT14:Subadditivity-Matrix}
formulated subadditivity properties for tracial entropy-like
quantities, including the trace variance~\eqref{eqn:matrix-var}.
They showed that these properties imply some
Sobolev and modified log-Sobolev-type inequalities for random matrices,
and they obtained some restricted matrix concentration inequalities.
Some of the partial results from~\cite{CT14:Subadditivity-Matrix}
were completed in~\cite{HZ15:Characterization-Matrix,PV15:Joint-Convexity,CH16:Characterizations-Matrix}.

Cheng et al.~\cite{CHT17:Exponential-Decay} developed a framework for studying Markov
processes in the mixed classical-quantum setting (and beyond),
and they showed an equivalence between tracial log-Sobolev
inequalities and exponential ergodicity of the trace entropy.
Further results and implications for concentration appear in~\cite{CH19:Matrix-Poincare}.
At present, we do not have a full picture of the relationships
between matrix functional inequalities and matrix concentration.

Van Handel (personal communication) has pointed out that we can derive
nonlinear matrix concentration for functions of a log-concave measure
with a strongly convex potential by combining Pisier's method~\cite{Pis86:Probabilistic-Methods}, the (sharp)
noncommutative Khintchine inequality~\cite{Buc01:Operator-Khintchine,Tro16:Expected-Norm},
and Caffarelli's contraction theorem~\cite{Caf00:Monotonicity-Properties}.
This approach gives subgaussian concentration, which is better than we can obtain via
the trace Poincar{\'e} inequality for log-concave measures, but
it apparently does not extend beyond this setting.

\subsection{From Poincar{\'e} to concentration}

It has been recognized for about 40 years that Poincar{\'e} inequalities
imply exponential concentration. Gromov \& Milman~\cite[Thm.~4.1]{GM83:Topological-Application}
prove such a theorem in the context of Riemannian manifolds.
The standard argument, a recursive estimate of the moment generating function, is attributed
to Aida \& Stroock~\cite{AS94:Moment-Estimates}.  For a textbook presentation,
see~\cite[Sec.~3.6]{BLM13:Concentration-Inequalities}.

Consider a matrix Poincar{\'e} inequality of the form
\begin{equation} \label{eqn:matrix-poincare}
\Var_{\mu}[\mtx{f}]	
\psdle \alpha \cdot \dirform(\mtx{f}).
\end{equation}
The argument in Proposition~\ref{prop:equiv} shows that
this matrix Poincar{\'e} inequality~\eqref{eqn:matrix-poincare}
is also equivalent to a scalar Poincar{\'e} inequality
with the same constant $\alpha$.
The papers~\cite{CH16:Characterizations-Matrix,CHT17:Exponential-Decay,CH19:Matrix-Poincare}
demonstrate that~\eqref{eqn:matrix-poincare} leads to some inequalities for
the matrix variance~\eqref{eqn:matrix-var} and its trace,
but these approaches do not lead to matrix concentration
inequalities like Theorem~\ref{thm:main}.

Aoun, Banna, and Youssef~\cite{ABY19:Matrix-Poincare} have recently
shown that the matrix Poincar{\'e} inequality~\eqref{eqn:matrix-poincare}
does imply exponential concentration of a random matrix about
its mean with respect to the $\ell_2$ operator norm.
Modulo constants, their result is equivalent with the tail bound~\eqref{eqn:main-tail},
but it is weaker than the bounds in Theorem~\ref{thm:poly-moment}.
The proof in~\cite{ABY19:Matrix-Poincare} is a direct analog of the argument of Aida \& Stroock~\cite{AS94:Moment-Estimates}.
But, in the matrix setting, the recursive estimate requires some heavy lifting.
Another contribution of the paper~\cite{ABY19:Matrix-Poincare} is to establish that some
particular matrix-valued Markov processes satisfy~\eqref{eqn:matrix-poincare}.
Nevertheless, Proposition~\ref{prop:equiv} indicates that no additional
effort is required for this end.

Our approach is similar in spirit to the work of Aoun et al.~\cite{ABY19:Matrix-Poincare},
but we use a symmetrization argument to avoid the difficult recursion.
For related work in the scalar setting, see~\cite[Sec.~4]{BL97:Poincares-Inequalities}.
In a companion paper~\cite{HT20:Matrix-Concentration}, we
show that \emph{local} Poincar{\'e} inequalities lead to much
stronger ergodicity and concentration properties.  The theory
in the companion paper is significantly more involved than
the development here, so we have chosen to separate them.
For results on (local) Poincar{\'e} inequalities in
noncommutative probability spaces,
see Junge \& Zeng~\cite{JZ15:Noncommutative-Martingale}.

\section{Subadditivity}

To control the variance of a function of several independent variables,
it is helpful to understand the influence of each individual variable.
As in the scalar setting, the matrix variance can be bounded by
a sum of conditional variances.  We can control each conditional
variance by a conditional application of a trace Poincar{\'e} inequality.
For simplicity, we focus on the case that is relevant to our proof,
but these results hold in greater generality
(products of more than two spaces that may carry different measures).
Some of the material in this section is drawn
from~\cite{CT14:Subadditivity-Matrix,CH16:Characterizations-Matrix}.
See~\cite{Led96:Talagrands-Deviation} for some of the classic
work on subadditivity and concentration. 

\subsection{Influence of a coordinate}

Consider the product space $(\Omega^2, \mu \otimes \mu)$.
We want to study how individual coordinates affect the
behavior of a matrix-valued function $\mtx{g} : \Omega^2 \to \Sym_d$.

First, introduce notation for the expectation of the function
with respect to each coordinate:
\begin{align*}
\E_{1}[ \mtx{g} ] (z_2) := \E_{Z \sim \mu}[ \mtx{g}(Z, z_2) ] \in \Sym_d \quad\text{for all $z_1 \in \Omega$;} \\
\E_{2}[ \mtx{g} ] (z_1) := \E_{Z \sim \mu}[ \mtx{g}(z_1, Z) ] \in \Sym_d \quad\text{for all $z_2 \in \Omega$.}
\end{align*}
The coordinate-wise variance is the positive-semidefinite random matrix
$$
\Var_i[ \mtx{g} ] := \E_{i} \big[ (\mtx{g} - \E_i \mtx{g} )^2 \big] \in \Sym_d^+
	\quad\text{for $i = 1, 2$.}
$$
This matrix reflects the fluctuation in the $i$th coordinate,
with the other coordinate held fixed.

Similarly, we can introduce the coordinate-wise carr{\'e} du champ operator
and Dirichlet form:
\begin{align}
\mtx{\Gamma}_1(\mtx{g})(z_1, z_2) &:= \lim_{t \downarrow 0} \frac{1}{2t}
	\E \big[ \big( \mtx{g}(Z_t, z_2) - \mtx{g}(Z_0, z_2) \big)^2 \, \big| \, Z_0 = z_1 \big] \in \Sym_d^+; \label{eqn:cond-carre} \\
\dirform_1(\mtx{g})(z_2) &:= \lim_{t \downarrow 0} \frac{1}{2t} \E_{Z \sim \mu} 
	\big[ \big( \mtx{g}(Z_t, z_2) - \mtx{g}(Z_0, z_2) \big)^2 \, \big| \, Z_0 = Z \big] \in \Sym_d^+. \label{eqn:cond-dirichlet}
\end{align}
As usual, $(Z_t : t \geq 0)$ is a realization of the Markov process with initial
value $Z_0$. We make analogous definitions for the second coordinate $i = 2$.

Last, we extend the carr{\'e} du champ operator and the Dirichlet form
to bivariate functions:
\begin{align}
\mtx{\Gamma}(\mtx{g})
	&:= \mtx{\Gamma}_1(\mtx{g})
	+ \mtx{\Gamma}_2(\mtx{g}) \in \Sym_d^+;  \label{eqn:bivar-carre} \\
\dirform(\mtx{g})
	&:= \Expect_{\mu \otimes \mu}\big[ \dirform_1(\mtx{g})
	+ \dirform_2(\mtx{g}) \big]  \in \Sym_d^+.  \label{eqn:bivar-dirichlet}	
\end{align}
These formulas have a heuristic interpretation: the squared derivative
of a bivariate function is the sum of the squared partial derivatives.

\subsection{Trace variance is subadditive}

Observe that the trace variance is controlled
by the sum of the coordinate-wise variances.

\begin{fact}[Trace variance: Subadditivity] \label{fact:var-decomp}
Let $\mtx{g} : \Omega^2 \to \Sym_d$ be a matrix-valued function on the product
space $(\Omega^2, \mu \otimes \mu)$.  Then
$$
\trace \Var_{\mu \otimes \mu}[ \mtx{g} ] \leq \E_2 \trace \Var_1[ \mtx{g} ] + \E_1 \trace \Var_2[ \mtx{g} ]. 
$$
\end{fact}

This result is due to Chen \& Tropp~\cite{CT14:Subadditivity-Matrix},
who showed that other matrix functions are also subadditive.  Later,
Cheng \& Hsieh~\cite{CH16:Characterizations-Matrix} noticed that
an analogous result holds without the trace.  Similar decompositions
are also valid for functions on the $n$-fold product $(\Omega^n, \mu^{\otimes n})$.

\begin{proof} The proof is the same as in the scalar setting.
Writing $\E = \E_{\mu \otimes \mu}$ for the total expectation,
\begin{align*}
\E \Var_{\mu \otimes \mu}[ \mtx{g} ]
	:= \E  \big[ (\mtx{g} - \E \mtx{g})^2 \big]
	&= \E  \big[ (\mtx{g} - \E_1 \mtx{g})^2 + (\E_1 \mtx{g} - \E_{1} \E_{2} \mtx{g})^2 \big] \\
	&\psdle \E [ (\mtx{g} - \E_1 \mtx{g})^2 ]
	+ \E [ (\mtx{g} - \E_2 \mtx{g})^2 ] \\
	&= \E_2  \Var_1[ \mtx{g} ] + \E_1  \Var_2[ \mtx{g} ].
\end{align*}
The first line relies on orthogonality of the conditionally centered
random matrices.  The second line requires the operator convexity
of the square, applied conditionally.  Last, take the trace.
\end{proof}

\subsection{Trace Poincar{\'e} inequalities are subadditive}

If the Markov process satisfies a trace Poincar{\'e} inequality,
then the variance of a bivariate function
also satisfies a trace Poincar{\'e} inequality.

\begin{proposition}[Trace Poincar{\'e}: Subadditivity] \label{prop:poincare-subadd}
Suppose that the Markov process satisfies a trace Poincar{\'e} inequality~\eqref{eqn:trace-poincare}
with constant $\alpha$.  Let $\mtx{g} : \Omega^2 \to \Sym_d$ be a suitable bivariate matrix-valued function.
Then
$$
\trace \Var_{\mu \otimes \mu}[ \mtx{g} ] 	\leq \alpha \cdot \trace \dirform(\mtx{g}).
$$
\end{proposition}

\begin{proof}
Start with Fact~\ref{fact:var-decomp}.  Apply the trace Poincar{\'e}
inequality~\eqref{eqn:trace-poincare} coordinate-wise to control
each of the two coordinate-wise variances.  Last, introduce the definition~\eqref{eqn:bivar-dirichlet}
of the Dirichlet form for a bivariate function.
\end{proof}

\section{Chain rule inequality for the Dirichlet form}

The key new tool in our approach is a simple trace inequality for the matrix
Dirichlet form that shows how it interacts with composition.

\begin{proposition}[Chain rule inequality] \label{prop:chain-rule}
Enforce Assumption~\ref{ass:main}.
Let $\phi : \R \to \R$ be a scalar function whose squared derivative $\psi = (\phi')^2$ is convex.
Then
$$
\trace \dirform( \phi(\mtx{f})) 
	= \E_{\mu} \trace \mtx{\Gamma}(\phi(\mtx{f}))
	\leq \E_{\mu} \trace \big[ \mtx{\Gamma}(\mtx{f}) \, \psi(\mtx{f}) \big]
	\quad\text{for all suitable $\mtx{f} : \Omega \to \Sym_d$.}
$$
In particular,
$$
\trace \dirform( \phi(\mtx{g})) 
	= \E_{\mu \otimes \mu} \trace \mtx{\Gamma}(\phi(\mtx{g}))
	\leq \E_{\mu \otimes \mu} \trace  \big[\mtx{\Gamma}(\mtx{g}) \, \psi(\mtx{g}) \big]
	\quad\text{for all suitable $\mtx{g} : \Omega^2 \to \Sym_d$.}
$$
\end{proposition}

\noindent
The proof of Proposition~\ref{prop:chain-rule} consumes the rest
of this section.

For context, recall that the carr{\'e} du champ operator $\Gamma$ of a (scalar-valued, reversible) diffusion process
satisfies a chain rule~\cite[Sec.~1.11]{BGL14:Analysis-Geometry}:
$$
\Gamma( \phi(f) ) = \Gamma( f ) \, \phi'(f)^2
\quad\text{for smooth $f$ and $\phi$.}
$$
Proposition~\ref{prop:chain-rule} provides a substitute for this relation for an arbitrary reversible Markov process
that takes matrix values.  In exchange for the wider applicability,
we need some additional averaging (provided by the Dirichlet form);
we must restrict our attention to functions $\phi$ with a convexity property;
and the equality is relaxed to an inequality.
In the scalar case, Proposition~\ref{prop:chain-rule} is related
to the Stroock--Varopoulos inequality~\cite{Str84:Introduction-Theory,Var85:Hardy-Littlewood-Theory}.

\subsection{Mean-value inequality for trace functions}

The argument hinges on an elementary trace inequality for deterministic matrices. This result is obtained by lifting a numerical inequality
to self-adjoint matrices.  A very similar statement~\cite[Lem.~3.4]{MJCFT14:Matrix-Concentration}
animates the exchangeable pairs approach to matrix concentration,
which is motivated by work in the scalar setting~\cite{Cha07:Steins-Method}.

\begin{lemma}[Mean-value trace inequality] \label{lem:mvti}
Let $\mtx{A}, \mtx{B} \in \Sym_d$.  Let $\phi : \R \to \R$ be a scalar function
whose squared derivative $\psi = (\phi')^2$ is convex.  Then
$$
\trace \big[ \big(\phi(\mtx{A}) - \phi(\mtx{B})\big)^2 \big]
	\leq \frac{1}{2} \trace \big[ (\mtx{A} - \mtx{B})^2 \big( \psi(\mtx{A}) + \psi(\mtx{B}) \big) \big].
$$
\end{lemma}

\begin{proof}
Let $a, b \in \R$.  The fundamental theorem of calculus and Jensen's inequality together deliver the relations
\begin{align*}
\big(\phi(a) - \phi(b)\big)^2 &= (a-b)^2 \left[ \int_0^1 \diff{\tau} \, \phi'\big(\tau a + (1-\tau) b \big) \right]^2 \\
	&\leq (a-b)^2 \int_0^1 \diff{\tau} \,\psi\big( \tau a + (1-\tau) b\big) \\
	&\leq (a-b)^2 \int_0^1 \diff{\tau} \, \big[ \tau \psi(a) + (1-\tau) \psi(b) \big]
	= \frac{1}{2} (a-b)^2 \big(\psi(a) + \psi(b)\big).
\end{align*}
The generalized Klein inequality~\cite[Prop.~3]{Pet94:Survey-Certain} allows us to lift
this numerical fact to a trace inequality for matrices $\mtx{A}, \mtx{B} \in \Sym_d$.
\end{proof}

\subsection{Proof of Proposition~\ref{prop:chain-rule}}

The result follows from a short calculation.
First, we use the definition~\eqref{eqn:dirichlet-limit}
of the Dirichlet form as a limit:
\begin{align*}
\trace \dirform( \phi(\mtx{f}) )
	&= \lim_{t \downarrow 0} \frac{1}{2t} \E_{Z \sim \mu} \big[\trace\big[ \big(\phi(\mtx{f}(Z_t)) - \phi(\mtx{f}(Z_0)) \big)^2 \big] \, \big| \, Z_0 = Z \big] \\
	&\leq \lim_{t \downarrow 0} \frac{1}{4t} \E_{Z \sim \mu} \trace \big[ \big(\mtx{f}(Z_t) - \mtx{f}(Z_0) \big)^2 \big(\psi(\mtx{f}(Z_t)) + \psi(\mtx{f}(Z_0)) \big) \, \big| \, Z_0 =Z \big] \\
	&= \lim_{t \downarrow 0} \frac{1}{2t} \E_{Z \sim \mu} \trace \big[ (\mtx{f}(Z_t) - \mtx{f}(Z_0))^2 \, \psi(\mtx{f}(Z_0)) \, \big| \, Z_0 = Z \big] \\
	&= \E_{Z \sim \mu} \trace \Big[ \lim_{t \downarrow 0} \frac{1}{2t} \E \Big[ (\mtx{f}(Z_t) - \mtx{f}(Z_0))^2 \, \big| \, Z_0 = Z \Big] \, \psi(\mtx{f}(Z)) \Big] \\
	&= \E_{\mu} \trace \big[ \mtx{\Gamma}(\mtx{f}) \, \psi(\mtx{f}) \big].
\end{align*}
The inequality is Lemma~\ref{lem:mvti}.  To reach the third line, we use the
fact that $(Z_0, Z_t)$ is an exchangeable pair for each $t \geq 0$, a consequence of the reversibility
of the Markov process $(Z_t : t \geq 0)$ and the fact that $Z_0 \sim \mu$.
Last, we condition on the value of $Z_0$, invoke dominated convergence to pass the expectation through
the limit, and we apply the definition~\eqref{eqn:carre-limit} of the carr{\'e} du champ operator.

\section{Exponential moments}
\label{sec:exp-moments}

Our main technical result is a bound for the exponential moments of a general matrix-valued
function on the state space.  In contrast to the usual approach of bounding the moment generating
function, we will compute the expectation of a hyperbolic function.  
\begin{theorem}[Exponential moments] \label{thm:exp-moment}
Enforce Assumption~\ref{ass:main}.
Let $\mtx{f} : \Omega \to \Sym_d$ be a function with $\E_{\mu} \mtx{f} = \mtx{0}$.
For $\theta > 0$,
\begin{equation} \label{eqn:exp-moment}
\E_{\mu} \trace \cosh(\theta \mtx{f})
	\leq d \cdot \left[ 1 + \frac{\alpha \theta^2 \ntr \dirform(\mtx{f})}{(1 - \alpha v_{\mtx{f}} \theta^2/2)_+} \right].
\end{equation}
The variance proxy $v_{\mtx{f}}$ is defined in~\eqref{eqn:vf},
and $\ntr$ is the normalized trace.
\end{theorem}

The proof of Theorem~\ref{thm:exp-moment} occupies the rest of the section.
But first, we use this moment bound to derive our main result, Theorem~\ref{thm:main}.

\begin{proof}[Proof of Theorem~\ref{thm:main} from Theorem~\ref{thm:exp-moment}]
Without loss, assume that $\E_{\mu} \mtx{f} = \mtx{0}$.
We use the matrix moment method~\cite{AW02:Strong-Converse,Oli10:Sums-Random}:
\begin{align*}
\mathbbm{P}_{\mu} \left\{ \norm{ \mtx{f} } \geq \lambda \right\}
	&\leq \inf_{\theta > 0} \frac{1}{\cosh(\theta \lambda)} \cdot \E_\mu \cosh(\theta \norm{\mtx{f}})
	= \inf_{\theta > 0} \frac{2}{\econst^{\theta \lambda} + \econst^{-\theta \lambda}} \cdot \E_{\mu} \norm{ \smash{\cosh(\theta \mtx{f})} } \\
	&\leq \inf_{\theta > 0} \frac{2}{\econst^{\theta \lambda}} \cdot \E_{\mu} \trace \cosh(\theta \mtx{f})
	\leq 2d \cdot \inf_{\theta > 0} \econst^{-\theta \lambda} \cdot \left[ 1 + \frac{\alpha \theta^2 \ntr \dirform(\mtx{f})}{(1 - \alpha v_{\mtx{f}} \theta^2/2)_+} \right] \\
	&\leq 2d \cdot \econst^{- \lambda / \sqrt{\alpha v_{\mtx{f}}}} \cdot \left[ 1 + \frac{2 \ntr \dirform(\mtx{f})}{ v_{\mtx{f}} } \right]
	\leq 6 d \cdot \econst^{- \lambda / \sqrt{\alpha v_{\mtx{f}}}}.
\end{align*}
The first inequality is Markov's.  The second relation is the spectral mapping theorem.
The $\ell_2$ operator norm of a positive-definite matrix is obviously bounded by its trace.
Then invoke Theorem~\ref{thm:exp-moment} to control the moment.
We have chosen $\theta = (\alpha v_{\mtx{f}})^{-1/2}$.  Last, we have noted that
$\ntr \dirform(\mtx{f}) \leq \norm { \dirform(\mtx{f})} \leq v_{\mtx{f}}$.
To finish the proof of the tail bound~\eqref{eqn:main-tail},
make the change of variables $\lambda \mapsto \lambda \sqrt{\alpha v_{\mtx{f}}}$.

The expectation bound~\eqref{eqn:main-expect} follows when we integrate the tail bound~\eqref{eqn:main-tail},
taking into account that the probability is also bounded by one; we omit the details.
\end{proof}

\subsection{Proof of Theorem~\ref{thm:exp-moment}}

Our goal is to develop an exponential moment bound for a
function $\mtx{f} : \Omega \to \Sym_d$ that satisfies
$\E_{\mu} \mtx{f} = \mtx{0}$.  We will need to work
with both the hyperbolic sine and cosine,
passing between them using simple identities.
After writing this paper, we learned that this proof is
a matrix analog of an argument proposed by
Bobkov \& Ledoux~\cite[Sec.~4]{BL97:Poincares-Inequalities}.

\subsubsection{Symmetrization}

The first step is to symmetrize the function.  Let $Z, Z' \in \Omega$
be independent random variables, each with distribution $\mu$.
Since $\mtx{f}(Z')$ has mean zero, a conditional application
of Jensen's inequality yields
\begin{equation} \label{eqn:coshf-bd}
\E \trace \sinh^2(\theta \mtx{f}(Z))
	\leq \E \trace \sinh^2( \theta (\mtx{f}(Z) - \mtx{f}(Z')))
	=: \E \trace \sinh^2( \theta \mtx{g}(Z,Z')).
\end{equation}
Indeed, since $\sinh^2$ is convex, the function $\trace \sinh^2$ is also convex~\cite[Prop.~2]{Pet94:Survey-Certain}.
We have defined the antisymmetric function $\mtx{g}(z, z') = \mtx{f}(z) - \mtx{f}(z')$ for $z, z' \in \Omega$.

\subsubsection{From moments to variance}
\label{sec:exp-mom-var}

Next, we will write the expectation as a variance.
Consider the odd function $\phi(s) = \sinh(\theta s)$.
First, we claim that
$$
\E \sinh( \theta \mtx{g}(Z, Z') ) = \mtx{0}.
$$
Indeed, using the antisymmetry of $\mtx{g}$
and the oddness of $\phi$,
$$
\E \sinh( \theta \mtx{g}(Z, Z') ) = \E \sinh( - \theta \mtx{g}(Z', Z) )
	= - \E \sinh( \theta \mtx{g}(Z', Z) )
	= - \E \sinh( \theta \mtx{g}(Z, Z') ).
$$
The last identity holds because $(Z, Z')$ is exchangeable.

As an immediate consequence,
$$
\E_{\mu \otimes \mu} \trace \sinh^2( \theta \mtx{g} )
	= \E_{\mu \otimes \mu} \trace \sinh^2( \theta \mtx{g} ) - \trace\big[ \big(\E_{\mu \otimes \mu} \sinh(\theta \mtx{g}) \big)^2 \big]
	=  \trace \Var_{\mu \otimes \mu}[ \sinh(\theta \mtx{g}) ].
$$
The appearance of the variance gives us access to Poincar{\'e} inequalities.

\subsubsection{Poincar{\'e} inequality}

To continue, we apply Proposition~\ref{prop:poincare-subadd}, the trace Poincar{\'e}
inequality for bivariate functions:
\begin{equation} \label{eqn:exp-mom-pf-2}
\E_{\mu \otimes \mu} \trace \sinh^2( \theta\mtx{g} )
	\leq \alpha \cdot \trace \dirform( \sinh(\theta \mtx{g}) )
	\leq \alpha \theta^2 \cdot \E_{\mu \otimes \mu} \trace \big[ \mtx{\Gamma}(\mtx{g}) \cosh^2(\theta \mtx{g}) \big].
\end{equation}
The second inequality is the chain rule, Proposition~\ref{prop:chain-rule},
for the Dirichlet form.  To activate it, we note
that the squared derivative of $\phi(s) = \sinh(\theta s)$
is the convex function $\psi(s) = \theta^2 \cosh^2(\theta s)$.

\subsubsection{Moment comparison}

A moment comparison argument allows us to isolate the exponential moment.
Define the variance proxy of the bivariate function:
\begin{equation} \label{eqn:vg}
v_{\mtx{g}} := \norm{ \norm{ \mtx{\Gamma}(\mtx{g}) } }_{L_{\infty}(\mu \otimes \mu)}.
\end{equation}
Continuing from~\eqref{eqn:exp-mom-pf-2}, the identity $\cosh^2 = 1 + \sinh^2$ implies that
\begin{align*}
\E_{\mu \otimes \mu} \trace \sinh^2( \theta\mtx{g} )
	&\leq \alpha \theta^2 \cdot \E_{\mu \otimes \mu} \trace \mtx{\Gamma}(\mtx{g}) + \alpha \theta^2 \cdot \E_{\mu \otimes \mu} \trace\big[\mtx{\Gamma}(\mtx{g}) \sinh^2(\theta \mtx{g}) \big] \\
	&\leq \alpha \theta^2 \cdot \trace \dirform(\mtx{g}) + \alpha \theta^2 v_{\mtx{g}} \cdot \E_{\mu\otimes \mu} \trace \sinh^2(\theta \mtx{g}).
\end{align*}
The second inequality is just the usual operator-norm bound for the trace.
Rearrange this identity to arrive at
\begin{equation} \label{eqn:coshg-bd}
\E_{\mu \otimes \mu} \trace \sinh^2( \theta \mtx{g} )
	\leq \frac{\alpha \theta^2 \trace \dirform(\mtx{g})}{(1 - \alpha v_{\mtx{g}} \theta^2)_+}.
\end{equation}
It remains to revert to the original function $\mtx{f}$.

\subsubsection{Comparison of carr{\'e} du champs}

Let us compute the Dirichlet form $\dirform(\mtx{g})$ and the variance proxy $v_{\mtx{f}}$
in terms of the original function $\mtx{f}$.
To that end, observe that the coordinate-wise carr{\'e} du champ~\eqref{eqn:cond-carre}
satisfies
\begin{align*}
\mtx{\Gamma}_1(\mtx{g})(z,z')
	&= \E \lim_{t \downarrow 0} \frac{1}{2t} \big[ \big(\mtx{g}(Z_t,z') - \mtx{g}(Z_0,z')\big)^2 \, \big| \, Z_0 = z \big] \\
	&= \E \lim_{t \downarrow 0} \frac{1}{2t} \big[ \big(\mtx{f}(Z_t) - \mtx{f}(Z_0) \big)^2 \, \big| \, Z_0 = z \big]
	= \mtx{\Gamma}(\mtx{f})(z).
\end{align*}
A similar calculation reveals that $\mtx{\Gamma}_2(\mtx{g})(z,z') = \mtx{\Gamma}(\mtx{f})(z')$.
Thus, the bivariate carr{\'e} du champ~\eqref{eqn:bivar-carre} takes the form
\begin{equation} \label{eqn:bivar-carre-sym}
\mtx{\Gamma}(\mtx{g})(z,z') = \mtx{\Gamma}_1(\mtx{g})(z,z') + \mtx{\Gamma}_2(\mtx{g})(z,z')
	= \mtx{\Gamma}(\mtx{f})(z) + \mtx{\Gamma}(\mtx{f})(z').
\end{equation}
As a consequence, the Dirichlet form can be calculated as
\begin{equation} \label{eqn:dirform-gf}
\dirform(\mtx{g}) = \E_{\mu \otimes \mu} \mtx{\Gamma}(\mtx{g})
	= \E_{\mu \otimes \mu} \big[\mtx{\Gamma}(\mtx{f})(z) + \mtx{\Gamma}(\mtx{f})(z') \big]
	= 2\dirform(\mtx{f}).
\end{equation}
The variance proxy~\eqref{eqn:vg} of the bivariate function satisfies
\begin{equation} \label{eqn:vg-vf}
v_{\mtx{g}} = 	\norm{ \norm{ \mtx{\Gamma}(\mtx{g})(z,z') } }_{L_{\infty}(\mu \otimes \mu)}
		\leq 		\norm{ \norm{ \mtx{\Gamma}(\mtx{f})(z) } }_{L_{\infty}(\mu)}
	+ 	\norm{ \norm{ \mtx{\Gamma}(\mtx{f})(z') } }_{L_{\infty}(\mu)}
	= 2 v_{\mtx{f}}.
\end{equation}
The last relation is the definition~\eqref{eqn:vf} of the variance proxy $v_{\mtx{f}}$.

\subsubsection{Endgame}

Combining~\eqref{eqn:coshf-bd} and~\eqref{eqn:coshg-bd}, we see that
\begin{equation*}
\E_{\mu} \trace \sinh^2(\theta \mtx{f})
	\leq  \frac{\alpha \theta^2 \trace \dirform(\mtx{g})}{(1 - \alpha v_{\mtx{g}} \theta^2)_+}
	\leq \frac{2 \alpha \theta^2 \trace \dirform(\mtx{f})}{(1 - 2 \alpha v_{\mtx{f}} \theta^2)_+}.
\end{equation*}
We have also used the relations~\eqref{eqn:dirform-gf} and~\eqref{eqn:vg-vf}
from the last section.

To compete the proof of~\eqref{eqn:exp-moment},
invoke the identity
$\sinh^2(s) = (\cosh(2 s) - 1)/2$ to see that
\begin{align*}
\E_{\mu} \trace \cosh(2 \theta \mtx{f})
		\leq d + \frac{4 \alpha \theta^2 \trace \dirform(\mtx{f})}{(1 - 2 \alpha v_{\mtx{f}} \theta^2)_+}.
\end{align*}
Finally, introduce the normalized trace, $\ntr$,
and make the change of variables $\theta \mapsto \theta/2$ to arrive at
$$
\E_{\mu} \trace \cosh(\theta \mtx{f})
	\leq d \cdot \left[ 1 + \frac{\alpha \theta^2 \ntr \dirform(\mtx{f})}{(1 - \alpha v_{\mtx{f}} \theta^2 / 2)_+} \right].
$$
This is the stated result.

\section{Polynomial Moments}
\label{sec:poly-moments}

By a simple variation on the proof of Theorem~\ref{thm:exp-moment},
we can also obtain bounds for the polynomial moments of a random matrix.  

\begin{theorem}[Polynomial moments] \label{thm:poly-moment}
Enforce Assumption~\ref{ass:main}.
Let $\mtx{f} : \Omega \to \Sym_d$ be a function with $\E_{\mu} \mtx{f} = \mtx{0}$.
For $q = 1$ and $q \geq 1.5$,
\begin{equation} \label{eqn:poly-moment}
\big( \E_{\mu} \trace \abs{\mtx{f}}^{2q} \big)^{1/(2q)}
	\leq \sqrt{2 \alpha q^2} \cdot \big( \E_{\mu} \trace \mtx{\Gamma}(\mtx{f})^q \big)^{1/(2q)}.
\end{equation}
\end{theorem}

By combining Theorem~\ref{thm:poly-moment} with the moment method, we can obtain
probability bounds for $\norm{\mtx{f}}$.  Let us summarize how these results compare
with the main result, Theorem~\ref{thm:main}.  Observe that Theorem~\ref{thm:poly-moment}
gives a bound on the Schatten $2q$-norm of the random matrix $\mtx{f}(Z)$
in terms of the Schatten $2q$-norm of $\mtx{\Gamma}(\mtx{f})^{1/2}$.
We have the relation
\begin{equation} \label{eqn:poly-unif-1}
\big( \E_{\mu} \trace \mtx{\Gamma}(\mtx{f})^{q} \big)^{1/(2q)}
	\leq d^{1/(2q)} \cdot \norm{ \norm{ \mtx{\Gamma}(\mtx{f})(z) } }_{L_{\infty}(\mu)}^{1/2}
	= d^{1/(2q)} \cdot \sqrt{ v_{\mtx{f}} }.
\end{equation}
Therefore, Theorem~\ref{thm:poly-moment} potentially yields stronger bounds
than Theorem~\ref{thm:exp-moment}.

In particular, Theorem~\ref{thm:poly-moment} applies even when
$\mtx{\Gamma}(\mtx{f})$ is not uniformly bounded.
Example~\ref{ex:chaos} illustrates why this flexibility is valuable. 
In Section~\ref{sec:poly-var},
we show that slightly better polynomial moment bounds are possible
when $\mtx{\Gamma}(\mtx{f})$ is uniformly bounded.

\begin{remark}[Missing Parameters]
Theorem~\ref{thm:poly-moment} also holds for $q \in (1, 1.5)$, with an
extra factor of $\sqrt{2}$ on the right-hand side.  The proof uses
a variant of Proposition~\ref{prop:chain-rule} that only requires
the function $\psi$ to be monotone.
\end{remark}

\begin{example}[Gaussian Chaos]
\label{ex:chaos}
Consider the matrix Gaussian chaos
$$
\vct{f}(\vct{X}) = \sum_{i, j = 1}^n X_i X_j \, \mtx{A}_{ij}
\quad\text{where $\vct{X} \sim \gamma_n$ and $\mtx{A}_{ij} = \mtx{A}_{ji} \in \Sym_d$.}
$$
To bound the trace moments of $f(\vct{X})$,
observe that the carr{\'e} du champ takes the form
$$
\mtx{\Gamma}(\mtx{f})(\vct{x})
	= \sum_{i=1}^n ( \partial_i \mtx{f}(\vct{x}) )^2
	= 4 \sum_{i=1}^n \left( \sum_{j=1}^n x_j \mtx{A}_{ij} \right)^2.
$$
Evidently, $\mtx{\Gamma}(\mtx{f})$ is unbounded,
so Theorem~\ref{thm:main} does not apply.
But Theorem~\ref{thm:poly-moment} yields
\begin{equation} \label{eqn:gauss-chaos-bd}
\big( \E_{\gamma_n} \trace \abs{\mtx{f}}^{2q} \big)^{1/(2q)}
	\leq \sqrt{8 q^2} \cdot \left(
	\E_{\gamma_n} \trace \left[ \sum_{i=1}^n \left( \sum_{j=1}^n X_j \mtx{A}_{ij} \right)^2 \right]^q \right)^{1/(2q)}.
\end{equation}
We have used the fact that the Poincar{\'e} constant of the OU process is $\alpha = 1$.
Further bounds can be obtained by applying Theorem~\ref{thm:poly-moment} repeatedly.
In the matrix setting, there are obstacles that prevent us from simplifying~\eqref{eqn:gauss-chaos-bd}
completely (related to the fact that the partial transpose operator is not completely bounded).

In the scalar case $d = 1$, we can obtain more transparent conclusions.  Consider the
scalar-valued Gaussian chaos
$$
f(\mtx{X}) = \sum_{i,j=1}^n X_i X_j a_{ij}.
$$
The most satisfying outcome takes place when $\mtx{A} = [ a_{ij} ]$ is positive semidefinite.  
In this case, the result~\eqref{eqn:gauss-chaos-bd} implies that
$$
\begin{aligned}
\big( \E_{\gamma_n} \abs{f}^{2q} \big)^{1/(2q)}
	&\leq \sqrt{8q^2} \cdot \left( \E_{\gamma_n} \left[ \sum_{i=1}^n \left( \sum_{j=1}^n X_j a_{ij} \right)^2 \right]^q \right)^{1/(2q)} \\
	&= \sqrt{8q^2} \cdot \left( \E_{\gamma_n} \left[ \sum_{j,k=1}^n X_j X_k (\mtx{A}^2)_{jk} \right]^q \right)^{1/(2q)} \\
	&\leq \sqrt{8 q^2 \norm{\mtx{A}}} \cdot \big(\E_{\gamma_n} \abs{f}^q\big)^{1/2q}
	\leq \sqrt{8 q^2 \norm{\mtx{A}}} \cdot \big(\E_{\gamma_n} \abs{f}^{2q} \big)^{1/4q}. 
\end{aligned}
$$
Solving, we obtain the bound
$$
\big( \E_{\gamma_n} \abs{f}^{2q} \big)^{1/(2q)}
	\leq 8 q^2 \norm{\mtx{A}}.
$$
This result gives a suboptimal estimate for large moments of the Gaussian chaos;
as $q \to \infty$, the moments should grow in proportion to $q \norm{\mtx{A}}$ rather than $q^2 \norm{\mtx{A}}$.
For example, see~\cite[Cor.~3.9]{LT91:Probability-Banach}.
\end{example}

\subsection{Proof of Theorem~\ref{thm:poly-moment}}

For a parameter $q = 1$ or $q \geq 1.5$, we wish to estimate the Schatten $2q$-norm
of a function $\mtx{f} : \Omega \to \Sym_d$ that satisfies $\E_{\mu} \mtx{f} = \mtx{0}$.
The argument has the same structure as Theorem~\ref{thm:exp-moment}.

First, we symmetrize.  Let $Z,Z' \in \Omega$ be independent random variables,
each with distribution $\mu$.  Jensen's inequality implies that
\begin{equation} \label{eqn:powf-bd}
\E \trace \abs{ \mtx{f} }^{2q}
	\leq \E \trace \abs{ \mtx{f}(Z) - \mtx{f}(Z') }^{2q}
	=: \E \trace \abs{\mtx{g}(Z,Z')}^{2q}.
\end{equation}
Since $\abs{\cdot}^{2q}$ is convex, the function $\trace \abs{\cdot}^{2q}$
is also convex~\cite[Prop.~2]{Pet94:Survey-Certain}.

Define the signed moment function $\phi(s) := \sgn(s) \cdot \abs{s}^{q}$,
which is odd.  Note that its squared derivative $\psi(s) := (\phi'(s))^2 = q^2 \abs{s}^{2(q-1)}$ is convex. 
Since $\mtx{g}(z,z') = \mtx{f}(z) - \mtx{f}(z')$ is antisymmetric,
\begin{align*}
\E_{\mu \otimes \mu} \trace \abs{\mtx{g}}^{2q}
	&= \E_{\mu \otimes \mu} \trace \phi(\mtx{g})^2 \\
	&= \E_{\mu \otimes \mu} \trace \phi(\mtx{g})^2 - \trace \big[ \big( \E_{\mu \otimes \mu} \phi(\mtx{g}) \big)^2 \big]
	= \trace \Var_{\mu \otimes \mu}[ \phi(\mtx{g}) ].
\end{align*}
Apply the bivariate trace Poincar{\'e} inequality, Proposition~\ref{prop:poincare-subadd}:
$$
\trace \Var_{\mu \otimes \mu}[ \phi(\mtx{g}) ]
	\leq \alpha \cdot \trace \dirform( \phi(\mtx{g}) )
	\leq \alpha q^2 \cdot \E_{\mu \otimes \mu} \trace\big[ \mtx{\Gamma}(\mtx{g}) \, \psi(\mtx{g}) \big].
$$
The second bound is the chain rule inequality, Proposition~\ref{prop:chain-rule}.
In summary,
\begin{equation} \label{eqn:q-to-q-1}
\E_{\mu \otimes \mu} \trace \abs{\mtx{g}}^{2q}
	\leq \alpha q^2 \cdot \E_{\mu \otimes \mu} \trace \big[\mtx{\Gamma}(\mtx{g}) \abs{\mtx{g}}^{2(q-1)} \big].
\end{equation}
This formula allow us to perform a moment comparison.

To isolate the carr{\'e} du champ $\mtx{\Gamma}(\vct{g})$, invoke
H{\"o}lder's inequality for the Schatten norms: $$
\E_{\mu \otimes \mu} \trace \abs{\mtx{g}}^{2q}
	\leq \alpha q^2 \big( \E_{\mu \otimes \mu} \trace \mtx{\Gamma}(\mtx{g})^q \big)^{1/q}
	 \big( \E_{\mu \otimes \mu} \trace \abs{\mtx{g}}^{2q} \big)^{(q-1)/q}.
$$
Rearrange the last display, and use the initial bound~\eqref{eqn:powf-bd} to arrive at
$$
\big( \E_{\mu \otimes \mu} \trace \abs{\mtx{f}}^{2q} \big)^{1/(2q)}
	\leq \big( \E_{\mu \otimes \mu} \trace \abs{\mtx{g}}^{2q} \big)^{1/(2q)}
	\leq \sqrt{\alpha q^2} \cdot \big( \E_{\mu \otimes \mu} \trace \mtx{\Gamma}(\mtx{g})^q \big)^{1/(2q)}.
$$
To finish the proof of~\eqref{eqn:poly-moment}, recall the expression~\eqref{eqn:bivar-carre-sym}
for the carr{\'e} du champ $\mtx{\Gamma}(\mtx{g})$.  Thus,
$$
\big( \E_{\mu \otimes \mu} \trace \mtx{\Gamma}(\mtx{g})^q \big)^{1/q}
	= \big( \E_{\mu \otimes \mu} \trace \abs{ \mtx{\Gamma}(\mtx{f})(z) + \mtx{\Gamma}(\mtx{f})(z') }^q \big)^{1/q}
	\leq 2 \big( \E_{\mu} \trace \mtx{\Gamma}(\mtx{f})^q  \big)^{1/q}.
$$
This point follows from the triangle inequality. The argument is complete.

\subsection{A variant of the argument}
\label{sec:poly-var}

The \emph{intrinsic dimension} of a positive-semidefinite matrix is $$
\intdim(\mtx{A}) := \frac{\trace(\mtx{A})}{\norm{\mtx{A}}}
\quad\text{for $\mtx{A} \in \Sym_d^+$.}
$$
We also set $\intdim(\mtx{0}) = 0$.
For a nonzero matrix $\mtx{A}$, the intrinsic dimension
satisfies $1 \leq \intdim(\mtx{A}) \leq \rank(\mtx{A})$.
It can be interpreted as a continuous measure of the rank.

Suppose that $q$ is a natural number.
If we use the uniform bound~\eqref{eqn:vg} for the carr{\'e} du champ
instead of H{\"o}lder's inequality, we can apply the bound~\eqref{eqn:q-to-q-1}
iteratively to obtain
$$
\E_{\mu \otimes \mu} \trace \abs{\mtx{g}}^{2q}
	\leq \intdim(\dirform(\mtx{g})) \cdot \alpha^q q! \cdot v_{\mtx{g}}^q.
$$
The latter estimate improves over the uniform inequality that follows from
Theorem~\ref{thm:poly-moment} and~\eqref{eqn:poly-unif-1}.

\section*{Acknowledgments}

Ramon Van Handel offered valuable feedback on a preliminary version of this work,
and we are grateful to him for the proof of Proposition~\ref{prop:equiv}.
DH was funded by NSF grants DMS-1907977 and DMS-1912654.
JAT gratefully acknowledges funding from ONR awards N00014-17-12146 and N00014-18-12363,
and he would like to thank his family for their support in these difficult times.

\bibliographystyle{myalpha}

\begin{thebibliography}{MJC{\etalchar{+}}14}

\bibitem[ABY19]{ABY19:Matrix-Poincare}
R.~Aoun, M.~Banna, and P.~Youssef.
\newblock Matrix {P}oincar{\'e} inequalities and concentration, 2019.

\bibitem[AS94]{AS94:Moment-Estimates}
S.~Aida and D.~Stroock.
\newblock Moment estimates derived from {P}oincar\'{e} and logarithmic
  {S}obolev inequalities.
\newblock {\em Math. Res. Lett.}, 1(1):75--86, 1994.

\bibitem[AW02]{AW02:Strong-Converse}
R.~Ahlswede and A.~Winter.
\newblock Strong converse for identification via quantum channels.
\newblock {\em IEEE Trans. Inform. Theory}, 48(3):569--579, 2002.

\bibitem[BGL14]{BGL14:Analysis-Geometry}
D.~Bakry, I.~Gentil, and M.~Ledoux.
\newblock {\em Analysis and geometry of {M}arkov diffusion operators}, volume
  348 of {\em Grundlehren der Mathematischen Wissenschaften [Fundamental
  Principles of Mathematical Sciences]}.
\newblock Springer, Cham, 2014.

\bibitem[BL97]{BL97:Poincares-Inequalities}
S.~Bobkov and M.~Ledoux.
\newblock Poincar\'{e}'s inequalities and {T}alagrand's concentration
  phenomenon for the exponential distribution.
\newblock {\em Probab. Theory Related Fields}, 107(3):383--400, 1997.

\bibitem[BLM13]{BLM13:Concentration-Inequalities}
S.~Boucheron, G.~Lugosi, and P.~Massart.
\newblock {\em Concentration inequalities}.
\newblock Oxford University Press, Oxford, 2013.
\newblock A nonasymptotic theory of independence, With a foreword by Michel
  Ledoux.

\bibitem[BT06]{BT06:Modified-Logarithmic}
S.~G. Bobkov and P.~Tetali.
\newblock Modified logarithmic {S}obolev inequalities in discrete settings.
\newblock {\em J. Theoret. Probab.}, 19(2):289--336, 2006.

\bibitem[Buc01]{Buc01:Operator-Khintchine}
A.~Buchholz.
\newblock Operator {K}hintchine inequality in non-commutative probability.
\newblock {\em Math. Ann.}, 319(1):1--16, 2001.

\bibitem[Caf00]{Caf00:Monotonicity-Properties}
L.~A. Caffarelli.
\newblock Monotonicity properties of optimal transportation and the {FKG} and
  related inequalities.
\newblock {\em Comm. Math. Phys.}, 214(3):547--563, 2000.

\bibitem[CH16]{CH16:Characterizations-Matrix}
H.-C. Cheng and M.-H. Hsieh.
\newblock Characterizations of matrix and operator-valued {$\Phi$}-entropies,
  and operator {E}fron-{S}tein inequalities.
\newblock {\em Proc. A.}, 472(2187):20150563, 20, 2016.

\bibitem[CH19]{CH19:Matrix-Poincare}
H.-C. Cheng and M.-H. Hsieh.
\newblock Matrix {P}oincar\'{e}, {$\Phi$}-{S}obolev inequalities, and quantum
  ensembles.
\newblock {\em J. Math. Phys.}, 60(3):032201, 16, 2019.

\bibitem[Cha07]{Cha07:Steins-Method}
S.~Chatterjee.
\newblock Stein's method for concentration inequalities.
\newblock {\em Probab. Theory Related Fields}, 138(1-2):305--321, 2007.

\bibitem[CHT17]{CHT17:Exponential-Decay}
H.-C. Cheng, M.-H. Hsieh, and M.~Tomamichel.
\newblock Exponential decay of matrix {$\Phi$}-entropies on {M}arkov semigroups
  with applications to dynamical evolutions of quantum ensembles.
\newblock {\em J. Math. Phys.}, 58(9):092202, 24, 2017.

\bibitem[CT14]{CT14:Subadditivity-Matrix}
R.~Y. Chen and J.~A. Tropp.
\newblock Subadditivity of matrix {$\phi$}-entropy and concentration of random
  matrices.
\newblock {\em Electron. J. Probab.}, 19:no. 27, 30, 2014.

\bibitem[Dav69]{Dav69:Quantum-Stochastic}
E.~B. Davies.
\newblock Quantum stochastic processes.
\newblock {\em Comm. Math. Phys.}, 15:277--304, 1969.

\bibitem[GKS20]{GKS20:Scalar-Poincare}
A.~Garg, T.~Kathuria, and N.~Srivastava.
\newblock Scalar {P}oincar{\'e} implies matrix {P}oincar{\'e}.
\newblock arXiv ma.PR 2006.09567, June 2020.

\bibitem[GLSS18]{GLSS18:Matrix-Expander}
A.~Garg, Y.~T. Lee, Z.~Song, and N.~Srivastava.
\newblock A matrix expander {C}hernoff bound.
\newblock In {\em S{TOC}'18---{P}roceedings of the 50th {A}nnual {ACM} {SIGACT}
  {S}ymposium on {T}heory of {C}omputing}, pages 1102--1114. ACM, New York,
  2018.

\bibitem[GM83]{GM83:Topological-Application}
M.~Gromov and V.~D. Milman.
\newblock A topological application of the isoperimetric inequality.
\newblock {\em Amer. J. Math.}, 105(4):843--854, 1983.

\bibitem[Gro75]{Gro75:Hypercontractivity-Logarithmic}
L.~Gross.
\newblock Hypercontractivity and logarithmic {S}obolev inequalities for the
  {C}lifford {D}irichlet form.
\newblock {\em Duke Math. J.}, 42(3):383--396, 1975.

\bibitem[HS19]{HS19:Modified-Log-Sobolev}
J.~Hermon and J.~Salez.
\newblock Modified log-{S}obolev inequalities for strong-{R}ayleigh measures,
  2019.

\bibitem[HT20]{HT20:Matrix-Concentration}
D.~Huang and J.~A. Tropp.
\newblock Nonlinear matrix concentration via semigroup methods.
\newblock Manuscript, June 2020.

\bibitem[HZ15]{HZ15:Characterization-Matrix}
F.~Hansen and Z.~Zhang.
\newblock Characterisation of matrix entropies.
\newblock {\em Lett. Math. Phys.}, 105(10):1399--1411, 2015.

\bibitem[JZ15]{JZ15:Noncommutative-Martingale}
M.~Junge and Q.~Zeng.
\newblock Noncommutative martingale deviation and {P}oincar{\'e} type
  inequalities with applications.
\newblock {\em Probab. Theory Related Fields}, 161(3-4):449--507, 2015.

\bibitem[KT13]{KT13:Quantum-Logarithmic}
M.~J. Kastoryano and K.~Temme.
\newblock Quantum logarithmic {S}obolev inequalities and rapid mixing.
\newblock {\em J. Math. Phys.}, 54(5):052202, 30, 2013.

\bibitem[Led01]{Led01:Concentration-Measure}
M.~Ledoux.
\newblock {\em The concentration of measure phenomenon}, volume~89 of {\em
  Mathematical Surveys and Monographs}.
\newblock American Mathematical Society, Providence, RI, 2001.

\bibitem[Led97]{Led96:Talagrands-Deviation}
M.~Ledoux.
\newblock On {T}alagrand's deviation inequalities for product measures.
\newblock {\em ESAIM Probab. Statist.}, 1:63--87, 1995/97.

\bibitem[Lin76]{Lin76:Generators-Quantum}
G.~Lindblad.
\newblock On the generators of quantum dynamical semigroups.
\newblock {\em Comm. Math. Phys.}, 48(2):119--130, 1976.

\bibitem[LP86]{LP86:Inegalites-Khintchine}
F.~Lust-Piquard.
\newblock In\'{e}galit\'{e}s de {K}hintchine dans {$C_p\;(1<p<\infty)$}.
\newblock {\em C. R. Acad. Sci. Paris S\'{e}r. I Math.}, 303(7):289--292, 1986.

\bibitem[LT11]{LT91:Probability-Banach}
M.~Ledoux and M.~Talagrand.
\newblock {\em Probability in {B}anach spaces}.
\newblock Classics in Mathematics. Springer-Verlag, Berlin, 2011.
\newblock Isoperimetry and processes, Reprint of the 1991 edition.

\bibitem[MJC{\etalchar{+}}14]{MJCFT14:Matrix-Concentration}
L.~Mackey, M.~I. Jordan, R.~Y. Chen, B.~Farrell, and J.~A. Tropp.
\newblock Matrix concentration inequalities via the method of exchangeable
  pairs.
\newblock {\em Ann. Probab.}, 42(3):906--945, 2014.

\bibitem[MOZ98]{MOZ98:Dissipative-Dynamics}
A.~W. Majewski, R.~Olkiewicz, and B.~Zegarlinski.
\newblock Dissipative dynamics for quantum spin systems on a lattice.
\newblock {\em J. Phys. A}, 31(8):2045--2056, 1998.

\bibitem[M{\"u}n19]{Mun19:Li-Yau-Inequality}
F.~M{\"u}nch.
\newblock {L}i--{Y}au inequality under $cd(0,n)$ on graphs.
\newblock arXiv ma.DG 1909.10242, Sep. 2019.

\bibitem[Oli09]{Oli09:Concentration-Adjacency}
R.~I. Oliveira.
\newblock Concentration of the adjacency matrix and of the {L}aplacian in
  random graphs with independent edges, 2009.

\bibitem[Oli10]{Oli10:Sums-Random}
R.~I. Oliveira.
\newblock Sums of random {H}ermitian matrices and an inequality by {R}udelson.
\newblock {\em Electron. Commun. Probab.}, 15:203--212, 2010.

\bibitem[OZ99]{OZ99:Hypercontractivity-Noncommutative}
R.~Olkiewicz and B.~Zegarlinski.
\newblock Hypercontractivity in noncommutative {$L_p$} spaces.
\newblock {\em J. Funct. Anal.}, 161(1):246--285, 1999.

\bibitem[Pet94]{Pet94:Survey-Certain}
D.~Petz.
\newblock A survey of certain trace inequalities.
\newblock In {\em Functional analysis and operator theory ({W}arsaw, 1992)},
  volume~30 of {\em Banach Center Publ.}, pages 287--298. Polish Acad. Sci.
  Inst. Math., Warsaw, 1994.

\bibitem[Pis86]{Pis86:Probabilistic-Methods}
G.~Pisier.
\newblock Probabilistic methods in the geometry of {B}anach spaces.
\newblock In {\em Probability and analysis ({V}arenna, 1985)}, volume 1206 of
  {\em Lecture Notes in Math.}, pages 167--241. Springer, Berlin, 1986.

\bibitem[PMT16]{PMT16:Efron-Stein-Inequalities}
D.~Paulin, L.~Mackey, and J.~A. Tropp.
\newblock {E}fron--{S}tein inequalities for random matrices.
\newblock {\em Ann. Probab.}, 44(5):3431--3473, 2016.

\bibitem[PP14]{PP14:Concentration-Lipschitz}
R.~Pemantle and Y.~Peres.
\newblock Concentration of {L}ipschitz functionals of determinantal and other
  strong {R}ayleigh measures.
\newblock {\em Combin. Probab. Comput.}, 23(1):140--160, 2014.

\bibitem[PV15]{PV15:Joint-Convexity}
J.~Pitrik and D.~Virosztek.
\newblock On the joint convexity of the {B}regman divergence of matrices.
\newblock {\em Lett. Math. Phys.}, 105(5):675--692, 2015.

\bibitem[PX97]{PX97:Noncommutative-Martingale}
G.~Pisier and Q.~Xu.
\newblock Non-commutative martingale inequalities.
\newblock {\em Comm. Math. Phys.}, 189(3):667--698, 1997.

\bibitem[Rud99]{Rud99:Random-Vectors}
M.~Rudelson.
\newblock Random vectors in the isotropic position.
\newblock {\em J. Funct. Anal.}, 164(1):60--72, 1999.

\bibitem[Str84]{Str84:Introduction-Theory}
D.~W. Stroock.
\newblock {\em An introduction to the theory of large deviations}.
\newblock Universitext. Springer-Verlag, New York, 1984.

\bibitem[Tro11]{Tro11:Freedmans-Inequality}
J.~A. Tropp.
\newblock Freedman's inequality for matrix martingales.
\newblock {\em Electron. Commun. Probab.}, 16:262--270, 2011.

\bibitem[Tro12]{Tro12:User-Friendly}
J.~A. Tropp.
\newblock User-friendly tail bounds for sums of random matrices.
\newblock {\em Found. Comput. Math.}, 12(4):389--434, 2012.

\bibitem[Tro15]{Tro15:Introduction-Matrix}
J.~A. Tropp.
\newblock An introduction to matrix concentration inequalities.
\newblock {\em Foundations and Trends in Machine Learning}, 8(1-2):1--230,
  2015.

\bibitem[Tro16]{Tro16:Expected-Norm}
J.~A. Tropp.
\newblock The expected norm of a sum of independent random matrices: an
  elementary approach.
\newblock In {\em High dimensional probability {VII}}, volume~71 of {\em Progr.
  Probab.}, pages 173--202. Springer, [Cham], 2016.

\bibitem[Var85]{Var85:Hardy-Littlewood-Theory}
N.~T. Varopoulos.
\newblock Hardy-{L}ittlewood theory for semigroups.
\newblock {\em J. Funct. Anal.}, 63(2):240--260, 1985.

\bibitem[VH16]{VH16:Probability-High}
R.~Van~Handel.
\newblock {P}robability in {H}igh {D}imension.
\newblock {APC} 550 lecture notes, Princeton Univ., 2016.

\end{thebibliography}
\newcommand{\etalchar}[1]{$^{#1}$}

\end{document}